\documentclass[12pt]{amsart}
\usepackage{amsthm,amsfonts,latexsym,amssymb,cite,hyperref}

\newtheorem{proposition}{\bf{Proposition}}[section]

\newtheorem{remark}{\sc{Remark} }[section]
\newtheorem{definition}{\sc{Definition} }[section]
\newtheorem{corollary}{\bf{Corollary} }[section]

\begin{document}

\title[Explicit estimates to nonlinear radiation-type problems]{Explicit estimates for  solutions of nonlinear radiation-type problems}
\author{Luisa Consiglieri}
\address{Luisa Consiglieri, Independent Researcher Professor, 
 European Union}
\urladdr{\href{http://sites.google.com/site/luisaconsiglieri}{http://sites.google.com/site/luisaconsiglieri}}

\begin{abstract}
We establish the existence of weak solutions of a nonlinear radiation-type
boundary value problem for elliptic equation on divergence form with
discontinuous leading coefficient.
Quantitative estimates play a crucial role on the real
applications. Our objective is the derivation 
of explicit expressions of the involved constants   in the quantitative
estimates, the so-called absolute or universal bounds. The dependence on
the leading coefficient and on the size of the spatial domain is precise.
This work shows that the expressions of those constants
are not so elegant as we might expect.
\end{abstract}


\maketitle

\section{Introduction}

Thermal effects on steady-state
physical and technological models, whatever they are from 
mechanical engineering, electrochemistry,
biomedical engineering, to mention a few,
 appear as an additional elliptic equation
with a nonlinear radiation-type boundary condition into the coupled PDE system
under study \cite{zamm,aduf,lap,epj}.
 These form a boundary value problem constituted by
an elliptic quasilinear
second order equation in divergence form with the leading coefficient
depending on the spatial variable and on the solution itself.
The problem of determining radiative effects
provides an interesting special case of a conormal derivative
boundary value problem for an elliptic divergence structure equation
 \cite{laitii}.
Here, we deal with the radiation-type condition  on a part of the boundary,
and on the remaining part the Neumann condition is taken into account.
Stationary heat conduction equation with the
radiation boundary condition (fourth power law) has been studied in 
two-dimensional \cite{milka} and three-dimensional \cite{simon} Lipschitz
domains.

In the existence theory, the quantitative estimates  of solutions to a linear elliptic equation in divergence form, with bounded and measurable coefficient,
play a crucial role. Indeed,
they enjoy a large interest in the literature (see for instance
  \cite{ben,caz,gia93,gt,lu,stamp63-64},
and the references therein).
Most
mathematicians have bearing to keep abstract the universal bounds along one
whole work. The values of the intervener constants are simply carried out.
It is forgotten that their values are crucial on the real applications
and/or the numerical analysis (see \cite{ko} and the references therein)
of the problems under study. Our objective is to fill such gap.

The outline of the present paper is as  follows.
We begin by stating the problem under study and its functional
framework in the next section. The Hilbert case is studied in Section \ref{diri}.
We derive $L^q$ (Section \ref{lpc}),
 $L^\infty$ (Section \ref{linfty}), 
 and $W^{1,q}$ (under $L^1$-data in Section \ref{l1data}) estimates for 
weak solutions.
 Finally, a $W^{1,p}$-estimate $(p<n/(n-1)$ for the Green kernel
and a $W^{1,q}$-estimate for weak solutions of linear boundary value problem,
the so-called mixed Robin-Neumann problem,
are obtained in Sections \ref{secg}
and \ref{lqc}, respectively.
Lipschitz domains, discontinuous leading coefficient, and 
$L^1$-data are the three mathematical shortcomings from the physical models
on the real world. It is taken them into account that our results are
stated.
 
\section{Statement of the problem}

Set $\Omega$ a  domain (that is, connected open set) in 
$\mathbb{R}^n$ ($n\geq 2$) of class $C^{0,1}$, and bounded.
Its boundary $\partial\Omega$ is constituted by two disjoint
 open $(n-1)$-dimensional sets,
$\Gamma_N$ and $\Gamma$,
such that $\partial\Omega=\bar\Gamma_N\cup \bar\Gamma$.
We consider $\Gamma_N$ over which the Neumann 
boundary condition is taken into account,
and $\Gamma$ over which the radiative effects may occur.

We study the following boundary value problem, in the sense of distributions,
\begin{eqnarray}
-\nabla\cdot(   \mathsf{A}
\nabla u)=f-\nabla\cdot{\bf f}&\mbox{ in }&\Omega;\label{omega}\\
(\mathsf{A}\nabla u-{\bf f})\cdot{\bf n}+b(u)=h&\mbox{ on }&\Gamma;
\label{robin}\\
(\mathsf{A}\nabla u-{\bf f})\cdot{\bf n}=g&\mbox{ on }&\Gamma_N, \label{gama}
\end{eqnarray}
where $\bf n$ is the unit outward normal to the boundary $\partial\Omega$.
Whenever the $(n\times n)$-matrix of the leading coefficient
is $\mathsf{A}=aI$, where $a$ is a real function and $I$ denotes
the identity matrix,  the elliptic  equation stands for isotropic materials.
Our problem includes the conormal derivative boundary value problem.
For that, it is sufficient to consider
the 
 situation $\Gamma=\partial\Omega$
(or equivalently $\Gamma_N=\emptyset$). The problem (\ref{omega})-(\ref{gama})
is the so-called mixed 
Robin-Neumann problem if the boundary condition (\ref{robin})
is linear, i.e.
\begin{equation}\label{bstar}
b(u)=b_*u,\qquad\mbox{for some }b_*>0.
\end{equation}

Set for any $p,\ell\geq 1$
\[V_{p,\ell}:=\{ v\in
W^{1,p}(\Omega):\ v\in L^{\ell}(\Gamma)\}
\]
the Banach space endowed with the norm 
\[
\| v\|_{V_{p,\ell}}:=\|   v\|_{p,\Omega}+
\|\nabla  v\|_{p,\Omega}+\|v\|_{\ell,\Gamma}.
\]
For the sake of simplicity, we denote by the same designation $v$
the trace of a function  $v\in W^{1,1}(\Omega)$.
For $p>1$, the space $V_{p,\ell}$ is reflexive by arguments given in
\cite{dpz}.
Observe that $V_{p,\ell}$ is a Hilbert space  equipped with the inner product
only if $p=\ell=2$.
The above norm is equivalent  to
 \begin{equation}\label{norm}
\| v\|_{1,p,\ell}:=\|\nabla  v\|_{p,\Omega}+\|v\|_{\ell,\Gamma},
\end{equation}
due to a Poincar\'e inequality \cite[Corollary 3]{chakib}:
\[
\|v\|_{p,\Omega}\leq P_p\left(\sum_{i=1}^n
\|\partial_iv\|_{p,\Omega}+|\Gamma|^{1/p-1}
\left|\int_\Gamma v\mathrm{ds}\right|\right).
\]
Here $|\cdot|$ stands for the $(n-1)$-Lebesgue measure.
Throughout this work, the significance of $|\cdot|$ also 
 stands for the Lebesgue measure of a set of $\mathbb{R}^n$.

By trace theorem,
\begin{eqnarray*}
V_{p,\ell}=W^{1,p}(\Omega),\quad \mbox{if } 1\leq\ell< p(n-1)/(n-p);\\
V_{p,\ell}\subset_{\not=}W^{1,p}(\Omega),\quad \mbox{if } \ell>p(n-1)/(n-p).
\end{eqnarray*}

For $1<q<n$,
 the best constants of the Sobolev and trace inequalities are, respectively, 
 \cite{tale,bond}
\begin{eqnarray*}
S_q&=&\pi^{-1/2}n^{-1/q}\left({q-1\over n-q}\right)^{1-1/q}\left[
{\Gamma(1+n/2)\Gamma(n)\over \Gamma (n/q)\Gamma(1+n-n/q)}\right]^{1/n};
\\
K_q&=&\pi^{(1-q)/2}\left({q-1\over n-q}\right)^{q-1}\left[
{\Gamma\left({q(n-1)\over 2(q-1)}\right)\Big/ \Gamma \left({
n-1\over 2(q-1)}\right)}\right]^{(q-1)/(n-1)},
\end{eqnarray*}
where $\Gamma$ stands for the Gamma function.
For $1^*=n/(n-1)$,  there exists the limit  constant
$S_1=\pi^{-1/2}n^{-1}[\Gamma(1+n/2)]^{1/n}$ \cite{tale}.
Hence, we introduce $S_{q,\ell}=S_q\max\{1+ P_q2^{(n-1)(1-1/q)},
P_q |\Gamma |^{1/q-1/\ell}\}$ and $K_{q,\ell}=K_q \max\{1+ P_q
2^{(n-1)(1-1/q)},
P_q |\Gamma |^{1/q-1/\ell}\}$ that verify
\begin{eqnarray} \label{sob}
\|v\|_{nq/(n-q),\Omega}\leq S_{q,\ell}\|v\|_{1,q,\ell}
;\\  \label{sobt}
\|v\|_{(n-1)q/(n-q),\partial\Omega}\leq K_{q,\ell}\|v\|_{1,q,\ell}.
\end{eqnarray}

\begin{definition}\label{def1}
We say that $u\in V_{p,\ell}$ is a weak solution to (\ref{omega})-(\ref{gama}), 
if it verifies
\begin{eqnarray}\label{pbu}
\int_{\Omega}    ( \mathsf{A}\nabla u)\cdot
\nabla v \mathrm{dx}+\int_{\Gamma}b(u) v \mathrm{ds}
=\int_{\Omega}{\bf f}\cdot\nabla v \mathrm{dx}+\\
+\int_{\Omega}fv \mathrm{dx}
+\int_{\Gamma_N}gv \mathrm{ds}+\int_{\Gamma}hv \mathrm{ds},
 \quad\forall v\in V_{p',\ell},\nonumber
\end{eqnarray}
where  
 ${\bf f}\in {\bf L}^{p}(\Omega)$,  $f\in L^{t}(\Omega)$, with
 $t=pn/(n+p)$
if $p>n/(n-1)$ and any $t>1$ if $1<p\leq n/(n-1)$,
  $g\in L^{s}(\Gamma_N)$, with $s=p(n-1)/n$ 
if $p>n/(n-1)$ and any $s>1$ if  $1<p\leq n/(n-1)$,
 and $h\in L^{\ell/(\ell-1)}(\Gamma)$.
\end{definition}
 
All terms on the right hand side of  (\ref{pbu})
 have sense, since the following embeddings hold:
 \begin{eqnarray*}
  W^{1,q}(\Omega)\hookrightarrow C(\bar\Omega)\quad\mbox{ for $q=p'>n,$
 i.e. $p<n/(n-1)$};\\
  \left.\begin{array}{l}
   {W}^{1,q}(\Omega)
 \hookrightarrow { L}^{q^*}(\Omega)\\
   W^{1,q}(\Omega)\hookrightarrow L^{q_*}(\partial\Omega)
   \end{array}\right\}\quad\mbox{ for $q=p'<n,$
i.e. $p>n/(n-1)$},
\end{eqnarray*}
 with  $q^*=qn/(n-q)$ and $q_*=q(n-1)/(n-q)$
 being the critical Sobolev and trace exponents, respectively, and
 $p'$ accounts for the conjugate exponent $p'=p/(p-1)$.
We observe that $q^*>1$ is arbitrary if $q=n$.

\begin{remark}
We emphasize that the existence of equivalence between the differential
(\ref{omega})-(\ref{gama}) and variational (\ref{pbu}) formulations
is only available under sufficiently data.
For instance, the Green formula may be applied if
$\mathsf{A}\nabla u \in {\bf L}^p(\Omega)$ and $\nabla\cdot(\mathsf{A}
\nabla u)\in L^p(\Omega)$.
\end{remark}

Assume
\begin{description}
\item[(A)]
 $\mathsf{A}=[A_{ij}]_{i,j=1,\cdots,n}
\in [L^\infty(\Omega)]^{n\times n}$ is uniformly elliptic, and uniformly bounded:
\begin{eqnarray}\label{amin}
\exists  a_\#>0,&&
A_{ij}(x)\xi_i\xi_j\geq a_\#|\xi|^2,
\quad\mbox{ a.e. }x\in\Omega,\ \forall \xi\in\mathbb{R}^n;\\
\exists a^\#>0,&&\|\mathsf{A}\|_{\infty,\Omega}\leq a^\#,\label{amax}
\end{eqnarray}
under the summation convention over repeated indices.
\item[(B)]
 $b:\Omega\times \mathbb{R}\rightarrow \mathbb{R}$ is a Carath\'eodory function
such that it is strictly monotone with respect to the last variable, and it has 
the following $(\ell-1)$-growthness properties:
\begin{eqnarray}\label{bmin}
\exists  b_\#>0, &&b(x,T){\rm sign}(T)\geq  b_\#|T|^{\ell-1};\\
\exists b^\# >0,&&|b(x,T)|\leq b^\#|T|^{\ell-1},\label{bmax}
\end{eqnarray} 
for a.e. $x\in\Omega$, and for all $T\in \mathbb{R}$.
\end{description}

\begin{remark}
If $b(T)=|T|^{\ell-2}T$, for all $T\in\mathbb{R}$, the property of
strong monotonicity occurs with $b_\#=2^{(2-\ell)}$ \cite[Lemma 3.3]{dpz}.
\end{remark}

\section{$V_{2,\ell}$-solvability ($\ell\geq 2$)}
\label{diri}

We establish the existence and uniqueness of weak solution
as well as its quantitative estimate. Although their proof is quite standard,
the  explicit expression of the bound is unknown, as far as we known.
\begin{proposition}\label{exist}
Let   ${\bf f}\in  {\bf L}^{2}(\Omega)$,  $f\in L^{t}(\Omega)$, with
 $t=(2^*)'$, i.e. $t=2n/(n+2)$
if $n>2$ and any $t>1$ if $n=2$, $g\in L^{s}(\Gamma_N)$, with $s=2(n-1)/n$ 
if $n>2$ and any $s>1$ if $n=2$.
Under the assumptions (A)-(B), there exists $u \in V_{2,\ell}$ 
being a weak solution to (\ref{omega})-(\ref{gama}), i.e.
 solving (\ref{pbu}) for all $v\in  V_{2,\ell}$. Moreover,
 the following estimate holds
\begin{eqnarray}
\label{cotau}
 {a_\#\over 2}\|\nabla u\|_{2,\Omega}^2+
{ b_\#(\ell -1)\over\ell}\| u\|_{\ell,\Gamma}^\ell\leq {1\over 
2a_\#}\left(
\|{\bf f}\|_{2,\Omega}+\mathcal{F}_{n}(
\|f\|_{t,\Omega},\|g\|_{s,\Gamma_N})\right)^2+\\
+ {\ell-1\over\ell b_\#^{1/(\ell-1)}}
\left(\|h\|_{\ell /(\ell-1),\Gamma}
+\mathcal{H}_{n}(\|f\|_{t,\Omega}, \|g\|_{s,\Gamma_N})\right)^{\ell/(\ell-1)}
:=\mathcal{A},\nonumber
\end{eqnarray}
where $\mathcal{F}_n(A,B)=\mathcal{H}_n(A,B)=S_{2,\ell}A+K_{2,\ell}B$
if $n>2$,  $\mathcal{F}_2(A,B)=\mathcal{H}_2( |\Omega |^{1/t'}A,
 |\Omega |^{1/(2s')}B)$, and
$\mathcal{H}_2(A,B)= S_{2t/(3t-2),\ell} A+ K_{2s/(2s-1),\ell}B$ if $t<2$.
In particular,  if $t\geq 2=n$, the estimate (\ref{cotau})
holds with $\mathcal{F}_2(A,B)=\mathcal{H}_2( |\Omega |^{1/2}A,
 |\Omega |^{1/(2s')}B)$, and $\mathcal{H}_2(A,B)= S_{1,\ell}|\Omega |^{1/2-1/t}
A+K_{2s/(2s-1),\ell}B$. 
\end{proposition}
\begin{proof}
The existence and uniqueness of a weak solution
$u\in V_{2,\ell}$ is consequence of the Browder-Minty theorem,
since the functional $T:V_{2,\ell}\rightarrow (V_{2,\ell})'$ defined by
\[
T(v)=
\int_{\Omega}    ( \mathsf{A}\nabla u)\cdot
\nabla v \mathrm{dx}+\int_{\Gamma}b(u) v \mathrm{ds}
\]
is strictly
monotone, continuous, bounded and coercive.

 Taking $v=u\in V_{2,\ell}$ as a test function in (\ref{pbu}),
 using the H\"older inequality  we obtain
 \begin{eqnarray}\label{pbuab}
 {a_\#}\|\nabla u\|_{2,\Omega}^2+
{ b_\#}\| u\|_{\ell,\Gamma}^\ell\leq 
 \|{\bf f}\|_{2,\Omega}\|\nabla u\|_{2,\Omega}+\\
+\|h\|_{\ell/(\ell-1),\Gamma}\| u\|_{\ell,\Gamma}
 +\|f\|_{t,\Omega}\|u\|_{t',\Omega}
+\|g\|_{s,\Gamma_N}\|u\|_{s',\Gamma_N}.\nonumber
\end{eqnarray}

For $n>2$,  making use of (\ref{sob}) and (\ref{sobt}) with $q=2$, we get
 \begin{eqnarray*}
 {a_\#\over 2}\|\nabla u\|_{2,\Omega}^2+
{ b_\#\over\ell\,'}\| u\|_{\ell,\Gamma}^\ell\leq 
 {1\over 2a_\#}\left(\|{\bf f}\|_{2,\Omega}+
 S_{2,\ell}\|f\|_{t,\Omega}+K_{2,\ell}
\|g\|_{s,\Gamma_N}\right)^2+\\
+{1\over \ell\,' b_\#^{1/(\ell-1)}}
\left(\|h\|_{\ell/(\ell-1),\Gamma}+S_{2,\ell}\|f\|_{t,\Omega}+K_{2,\ell}
\|g\|_{s,\Gamma_N}\right)^{\ell/(\ell-1)}.
\end{eqnarray*}
Therefore, (\ref{cotau}) follows.

Consider the case of dimension $n=2$.
For $t,s>1$, using the H\"older inequality in (\ref{sob})
with $q=2t'/(t'+2)$ 
if $t'\geq 2$, and in (\ref{sobt}) for any $s>1$,
we have
\begin{eqnarray*}
\|u\|_{t',\Omega}\leq
  S_{{2t\over 3t-2},\ell}\| u\|_{1,2t/(3t-2),\ell}\leq
  S_{{2t\over 3t-2},\ell}\left(
|\Omega |^{1/t'}\|\nabla u\|_{2,\Omega}+\|u\|_{\ell,\Gamma}\right) ;\\
\|u\|_{s',\Gamma_N}\leq K_{{2s\over 2s-1},\ell}\|u\|_{1,2s/(2s-1),\ell} \leq
  K_{{2s\over 2s-1},\ell}\left(|\Omega |^{1/(2s')}\|\nabla u\|_{2,\Omega} 
+\|u\|_{\ell,\Gamma}\right) .\end{eqnarray*}
  Inserting the above inequalities in (\ref{pbuab}), it results in (\ref{cotau}).
  
  Finally, if $t> 2$, we have 
\[
\|u\|_{t',\Omega}\leq 
|\Omega |^{1/2-1/t}\|u\|_{2,\Omega}\leq
|\Omega |^{1/2-1/t} S_{1,\ell}\left(
|\Omega |^{1/2}\|\nabla u\|_{2,\Omega}+\|u\|_{\ell,\Gamma}\right).
\]
This concludes the proof of Proposition \ref{exist}.
\end{proof}

\begin{remark}\label{sol}
Proposition \ref{exist} remains valid if the assumption
$h\in L^{\ell/(\ell-1)}(\Gamma)$ is replaced by 
$h\in L^{s}(\Gamma)$, with the estimate (\ref{cotau}) being rewritten with
\begin{eqnarray}
\label{cotausol}
{\mathcal A}={1\over 2a_\#}\left(
\|{\bf f}\|_{2,\Omega}+\mathcal{F}_{n}(
\|f\|_{t,\Omega},\|g\|_{s,\Gamma_N}+\|h\|_{s,\Gamma})\right)^2+\\
+ {\ell-1\over\ell b_\#^{1/(\ell-1)}}
\left[\mathcal{H}_{n}(\|f\|_{t,\Omega}, \|g\|_{s,\Gamma_N}+
\|h\|_{s,\Gamma})\right]^{\ell\,'}.\nonumber
\end{eqnarray}
\end{remark}

\begin{corollary}\label{coruq}
Under the conditions of Proposition \ref{exist},
we have
\begin{eqnarray}
\label{cotausup}\qquad
\|u\|_{2p/(p-2),\Omega}\leq
  S_{{2pn\over 2p+n(p-2)},\ell}\left(
|\Omega |^{{1\over n}-{1\over p}}\left({2\mathcal{A}\over a_\#}\right)^{1/2}
+\left({\ell\;'\mathcal{A}\over b_\#}\right)^{1/\ell}\right);\\
\|u\|_{2s',\Omega}\leq
  K_{{2sn\over 2s+(n-1)(s-1)},\ell}\left(
|\Omega |^{s-n+1\over 2ns}\left({2\mathcal{A}\over a_\#}\right)^{1/2}
+\left({\ell\;'\mathcal{A}\over b_\#}\right)^{1/\ell}\right),
\label{cotaus}
\end{eqnarray}
for $p\geq n>2$, $p>n=2$, $s\geq n-1>1$, and $s>1$ ($n=2$).
\end{corollary}
\begin{proof}
Making use of (\ref{sob}) with $q=2pn/[2p+n(p-2)]$ if $p>2$,
and the H\"older inequality for $p\geq n$, we obtain
\begin{eqnarray*}
\|u\|_{2p/(p-2),\Omega}\leq
  S_{{2pn\over 2p+n(p-2)},\ell}\| u\|_{1,{2pn\over 2p+n(p-2)},\ell}\leq\\
\leq  S_{{2pn/[2p+n(p-2)]},\ell}\left(
|\Omega |^{1/n-1/p}\|\nabla u\|_{2,\Omega}+\|u\|_{\ell,\Gamma}\right).
\end{eqnarray*}
Applying (\ref{cotau}) in the above inequality, 
we conclude (\ref{cotausup}).

Making use of (\ref{sobt}) with $q=2sn/[2s+(n-1)(s-1)]$ if $s>1$,
and the H\"older inequality for $s\geq n-1$, we obtain
\begin{eqnarray*}
\|u\|_{2s/(s-1),\partial\Omega}\leq
  K_{{2sn\over 2s+(n-1)(s-1)},\ell}\| u\|_{1,{2sn\over 2s+(n-1)(s-1)},\ell}\leq\\
\leq  K_{{2sn/[2s+(n-1)(s-1)]},\ell}\left(
|\Omega |^{(s-n+1)/(2ns)}\|\nabla u\|_{2,\Omega}+\|u\|_{\ell,\Gamma}\right).
\end{eqnarray*}
Thus, (\ref{cotaus}) holds as before.
\end{proof}

\section{$L^{q}$-estimates ($q<2(n-1)p/[2(n-1)-p]$, 
$2< p<2(n-1)$)}
\label{lpc}

Section \ref{diri} ensures the existence of a weak solution,
$u\in L^{2p/(p-2)}(\Omega)$,  to (\ref{omega})-(\ref{gama})
in accordance with Definition \ref{def1}, only if $p\geq n>2$.
Let us improve that.

First, 
let us introduce  the Marcinkiewicz space, $L^*_p(\Omega)$, which is
Banach space of the measurable functions
that have finite the following norm \cite{gt}:
\[
\|v\|_{*,p,\Omega}:=\sup _{t>0} t |
\Omega[|v|>t]|^{1/p},
\]
for $p>1$ and $0<\varepsilon\leq p-1$, and $\Omega[|v|>t]:=\{x\in\Omega:
\ |v(x)|>t\}$.
Moreover, we recall the following property
\begin{equation}\label{star}
\|v\|_{p-\varepsilon,\Omega}\leq
 \left({p\over\varepsilon}\right)^{1/(p-\varepsilon)}|\Omega|^{
\varepsilon/[p(p-\varepsilon)]}
\|v\|_{*,p,\Omega},\quad\forall v\in L^*_p(\Omega).
\end{equation}

 We derive the explicit estimates
via the analysis of the decay of the level sets of the solution
\cite{stamp63-64}, extending the global estimate established in \cite{da}
of $(u,u|_{\partial\Omega})$ in $L^{np/(n-p)}(\Omega)\times
L^{(n-1)p/(n-p)}(\partial\Omega)$ if $f\in L^{p}(\Omega)$ with
$2\leq p<n$.
\begin{proposition}\label{qestimates}
Let $2<p,r<2(n-1)$, $u \in H^{1}(\Omega)$ 
be any weak solution to (\ref{omega})-(\ref{gama})
in accordance with Definition \ref{def1},
and (\ref{amin}) and (\ref{bmin}) be fulfilled. If 
 ${\bf f}\in  {\bf L}^{p}(\Omega)$,  $f\in L^{np/(p+n)}(\Omega)$, 
 $g\in L^{(n-1)p/n}(\Gamma_N)$, and  $h\in L^{r}(\Gamma)$,
 then we have, for every $q<2(n-1)/[n-2-2(n-1)\delta]:=Q$,
\begin{equation}\label{cotaps}
\|u\|_{q,\Omega}+\|u\|_{q,\partial\Omega}\leq \mathcal{K}_{q,\delta}
\left(\mathcal{B}+2|\bar\Omega[|u|>1]|^{{n-2\over 2(n-1)}-\delta}
\right),
\end{equation}
where $\delta=\min\{1/ 2-1/ p,1/2-1/r\}$, and
 the positive constants $\mathcal{K}$ and $\mathcal{B}$ are
\begin{eqnarray*}&&
\mathcal{K}_{q,\delta}=
2^{n-2\over n-2-2(n-1)\delta} 
\left({Q\over Q-q}\right)^{1/q}
\left( |\Omega|^{{1\over q}-{1\over Q}}+|\partial\Omega|^{
{1\over q}-{1\over Q}}\right);\\
\mathcal{B}&=&( |\Omega|^{n-2\over 2(n-1) n} S_{2,2}+K_{2,2})
\left[ \left({1\over a_\#}
 +{1\over \sqrt{a_\#b_\#}}\right)
(\|{\bf f}\|_{p,\Omega}+ C_{n,p,r})|\Omega|^{{1\over 2}-{1\over p}-\delta}
\right.\\
&&\qquad\left. +\left({1\over b_\#}
+{1\over  \sqrt{a_\#b_\#}}\right)
(\|h\|_{r,\Gamma}+ C_{n,p,r})
|\Gamma|^{1/2-1/r-\delta}
\right];\\&&
C_{n,p,r}=S_{p',r'}\|f\|_{np/(p+n),\Omega}+K_{p',r'}
 \|g\|_{(n-1)p/n,\Gamma_N}
,\quad\forall n\geq 2.
\end{eqnarray*}
\end{proposition}
\begin{proof}
Let $k\geq k_0=1$. Hence forth we use the notation
   $A(k)=\{x\in A:\ |u(x)|>k\}$, with the set $A$ being either
   $\Omega$, $\Gamma_N$, $\Gamma$, $\partial\Omega$ or $\bar \Omega$.
Choosing $v={\rm sign}(u)(|u|-k)^+={\rm sign}(u)\max\{|u|-k,0\}\in
 H^1(\Omega)$ as a 
test function in (\ref{pbu}), then $\nabla v=\nabla u\in {\bf L}^2(
 \Omega (k))$. 
Since $|u|>1$ a.e. on $\Gamma(k)$, 
taking  (\ref{amin}) and (\ref{bmin}) into account, we deduce
\begin{eqnarray}\label{varak}
a_\#\int_{ \Omega (k)}|\nabla u|^2\mathrm{dx}+
b_\#\int_{ \Gamma (k)}(|u|-k)^{2}\mathrm{ds}
\leq \\
\leq \|{\bf f}\|_{2,\Omega(k)}\|\nabla u\|_{2,\Omega(k)}
+\|f\|_{{np\over p+n},\Omega}\|(|u|-k)^+\|_{{np\over np-n-p},\Omega} 
+\nonumber\\+
\|g\|_{{(n-1)p\over n}, \Gamma_N}\|(|u|-k)^+\|_{{(n-1)p\over np-n-p},\Gamma_N}
+\|h\|_{2, \Gamma(k)}\||u|-k\|_{2,\Gamma(k)}.\nonumber
\end{eqnarray}
Using the H\"older inequality, it follows that ($ p,r>2$)
\begin{eqnarray*}
\|{\bf f}\|_{2,\Omega(k)}
\leq \|{\bf f}\|_{p,\Omega}|\Omega(k)|^{1/2-1/p};\\
\|h\|_{2,\Gamma(k)}
\leq \|h\|_{r,\Gamma}|\Gamma(k)|^{1/2-1/r}.
\end{eqnarray*}

Making use of (\ref{sob})-(\ref{sobt}) and
$(|u|-k)^+\in V_{p',r'}$ with $p'<2\leq n$ and $r'<2$, 
and  the H\"older inequality, we get
\begin{eqnarray*}
\|(|u|-k)^+\|_{{np\over n(p-1)-p},\Omega}
\leq S_{p',r'}\left(|\Omega(k)|^{{1/ p'}-{1/ 2}}
\|\nabla u\|_{2,\Omega(k)}+
\right. \\ \left.
+|\Gamma (k)|^{{1/ r'}-{1/2}}
\||u|-k\|_{2,\Gamma (k)}\right);\\
\|(|u|-k)^+\|_{{(n-1)p\over n(p-1)-p},\Gamma_N}
\leq K_{p',r'}\left(|\Omega(k)|^{{1/ p'}-{1/ 2}}
\|\nabla u\|_{2,\Omega(k)}+ \right. \\ \left.
+|\Gamma (k)|^{{1/r'}-{1/2}}
\||u|-k\|_{2,\Gamma (k)}\right). 
\end{eqnarray*}
Inserting last four inequalities into (\ref{varak}) we obtain
\begin{eqnarray}\label{energy}
{a_\#}\|\nabla u\|_{2,\Omega(k)}^2+{b_\#}\| |u|-k
\|_{2,\Gamma(k)}^2\leq {\left(
\|{\bf f}\|_{p,\Omega}+ C_{n,p,r}\right)^2
\over a_\#} |\Omega(k)|^{1-{2\over p}}\\
+{\left(
\|h\|_{r,\Gamma}+ C_{n,p,r}\right)^{2}
\over b_\#}|\Gamma(k)|^{1-2/r} 
,\qquad \forall p,r>2.\nonumber
\end{eqnarray}
 It results in
\begin{eqnarray}\label{nuakn}
\|( |u|-k)^+\|_{1,2,2}\leq \left[ \left({1\over a_\#}
 +{1\over \sqrt{a_\#b_\#}}\right)
{(
\|{\bf f}\|_{p,\Omega}+ C_{n,p,r})|\Omega(k)|^{1/2-1/p}
\over (|\Omega(k)|+|\Gamma(k)|)^\delta} +\right.\\
\left.+ \left({1\over b_\#}
 +{1\over \sqrt{a_\#b_\#}}\right){(
\|h\|_{r,\Gamma}+ C_{n,p,r})|\Gamma(k)|^{1/2-1/r}
\over (|\Omega(k)|+|\Gamma(k)|)^\delta} \right]|\bar\Omega(k)|^\delta.\nonumber
\end{eqnarray}

For  $ h > k > k_0$, we have
\begin{equation}\label{aalfa1}
(h-k)|\bar \Omega (h)|^{1/\alpha}\leq 
\||u|-k\|_{\alpha,\Omega(k)}+
\||u|-k\|_{\alpha,\partial\Omega(k)},\quad \forall\alpha\geq 1.
\end{equation}

Choosing $\alpha= 2(n-1)/(n-2)$,
 we use (\ref{sob}) and
(\ref{sobt}),  $(|u|-k)^+\in W^{1,n\alpha/(\alpha+n-1)}(\Omega)$,
 with $n\alpha/(\alpha+n-1)<n$, and the H\"older inequality since
$n\alpha/(\alpha+n-1)\leq 2$. Thus, we have
\begin{eqnarray*}
\||u|-k\|_{\alpha,\Omega(k)}
\leq  |\Omega|^{1\over n\alpha}S_{2,2}
\left(\|\nabla u\|_{2,\Omega(k)}+\||u|-k\|_{2,\Gamma(k)}\right);\\
\||u|-k\|_{\alpha,\partial\Omega (k)}
\leq K_{2,2}
\left(\|\nabla u\|_{2,\Omega(k)}+\||u|-k\|_{2,\Gamma(k)}\right).
\end{eqnarray*}

Applying (\ref{nuakn})-(\ref{aalfa1}), we find
\[ 
(h-k)|\bar \Omega (h)|^{1/ \alpha}
\leq \mathcal{B}
|\bar\Omega(k)|^{\delta}.   
\]
Observing that  $\beta=
\alpha\delta<1$ if and only if $p,r<2(n-1)$, we may appeal
 to \cite[Lemma 4.1 (iii)]{stamp63-64}, deducing
 \[
h^{\alpha/(1-\beta)} |\bar\Omega(h)|\leq 2^{\alpha/(1-\beta)^2}\left(
\mathcal{B}^{\alpha/(1-\beta)}+ (2k_0)^{\alpha/(1-\beta)}
|\bar\Omega (k_0)|\right).
\]
Considering $k_0=1$, using (\ref{star}) the claimed estimate
(\ref{cotaps}) holds.
\end{proof}

\section{$L^\infty$-estimates}
\label{linfty}

 In this section, we establish some maximum principles, by recourse to
 the De Giorgi technique \cite{stamp63-64},
and the Moser iteration technique \cite[pp. 189-190]{gt}.
New results are stated that provide $L^\infty$-estimates up to the
boundary under any space dimension $n\geq 2$.

\subsection{De Giorgi technique}

\begin{proposition}\label{max}
Let $p>n\geq 2$, $r>2(n-1)$,  $u \in H^{1}(\Omega)$ 
be any weak solution to (\ref{omega})-(\ref{gama})
in accordance with Definition \ref{def1},
and (\ref{amin}) and (\ref{bmin}) be fulfilled.
Under the hypotheses 
 ${\bf f}\in  {\bf L}^{p}(\Omega)$,  $f\in L^{np/(p+n)}(\Omega)$, 
 $g\in L^{(n-1)p/n}(\Gamma_N)$, and  $h\in L^{r}(\Gamma)$,
 we have
\begin{equation}\label{supess}
{\rm  ess } \sup_{\Omega\cup \partial\Omega}|u|\leq 1
+\left\{\begin{array}{ll}
2^{\gamma/(\gamma-{1\over 2}+{1\over n})}
(|\Omega|+|\partial\Omega|)^{\gamma-{1\over 2}+{1\over n}}
\mathcal{Z}_n&\mbox{ if }n>2\\
2^{\alpha\gamma+1/2\over \alpha\gamma-1/2}
(|\Omega|+|\partial\Omega|)^{\gamma-1/(2\alpha)}\mathcal{Z}_2&\mbox{ if }n=2
\end{array}\right.
\end{equation}
where $\alpha>1/(2\gamma)$,
$\gamma=\min\{1/ 2-1/ p,(1/2-1/r)(n-1)/n\}$,  and
 $\mathcal{Z}_n$ is
\begin{eqnarray*}
 \mathcal{Z}_n=( S_{2,2}+K_{2,2})
\left[ \left({1\over a_\#}
 +{1\over \sqrt{a_\#b_\#}}\right)
(\|{\bf f}\|_{p,\Omega}+ C_{n,p,r})|\Omega|^{{1\over 2}-{1\over p}-\gamma}
\right.\\
\left. +\left({1\over b_\#}
+{1\over  \sqrt{a_\#b_\#}}\right)
(\|h\|_{r,\Gamma}+ C_{n,p,r})
|\Gamma|^{1/2-1/r-\gamma n/(n-1)}
\right]\quad \mbox{ if }n>2;\\
 \mathcal{Z}_2=( S_{1,1}+K_{1,1})
\left[ \left({|\Omega|^{1/( 2\alpha)}\over a_\#}
 +{1\over \sqrt{a_\#b_\#}}\right)
(\|{\bf f}\|_{p,\Omega}+ C_{2,p,r})|\Omega|^{{1\over 2}-{1\over p}-\gamma}
\right.\\
\left. +\left({1\over b_\#}
+{|\Omega|^{1/( 2\alpha)}\over  \sqrt{a_\#b_\#}}\right)
(\|h\|_{r,\Gamma}+ C_{2,p,r})
|\Gamma|^{1/2-1/r-2\gamma }
\right],
\end{eqnarray*}
with  $C_{n,p,r}$ being given in Proposition \ref{qestimates}.
\end{proposition}
\begin{proof}
Arguing as in the proof of Proposition \ref{qestimates},
(\ref{energy}) holds.
Defining $\Sigma(k)=|\Omega(k)|+|\partial\Omega(k)|^{n/(n-1)}$, we have
\begin{equation}\label{aalfa2}
(h-k)|\Sigma (h)|^{1/\alpha}\leq 
\||u|-k\|_{\alpha,\Omega(k)}+
\||u|-k\|_{\alpha (n-1)/n,\partial\Omega(k)}:=I,
\end{equation}
for every $ h > k > k_0=1$ and $\alpha\geq n/(n-1)$.

Next, taking   $\alpha=2n/(n-2)$  if $n>2$
and any  $\alpha>1/(2\gamma)$
if $n=2$, we get $(|u|-k)^+\in W^{1,n\alpha/(\alpha+n)}(\Omega)$.
Thus, we use (\ref{sob}) and
(\ref{sobt}) with $n\alpha/(\alpha+n)<n$, and the H\"older inequality, obtaining
\begin{eqnarray*}
\||u|-k\|_{\alpha,\Omega(k)}
\leq S_{{n\alpha\over \alpha+n},
{n\alpha\over \alpha+n}}
\left(|\Omega|^{\sf{z}}\|\nabla u\|_{2,\Omega(k)}+\||u|-k\|_{2,\Gamma(k)}\right)
\Sigma(k)^{\sf{z}};\\
\||u|-k\|_{{\alpha(n-1)\over n},\partial\Omega (k)}
\leq K_{{n\alpha \over\alpha+n},{n\alpha\over \alpha+n}}
\left(|\Omega|^{\sf{z}}\|\nabla u\|_{2,\Omega(k)}+\||u|-k\|_{2,\Gamma(k)}\right)
\Sigma(k)^{\sf{z}},
\end{eqnarray*}
where $\sf{z}=0$ if $n>2$, and $\sf{z}=1/(2\alpha)$ if $n=2$.
Let us split these two situations.
\begin{description}
\item[Case $n>2$.]
Applying (\ref{energy}), we obtain
\begin{eqnarray*} 
I\leq (S_{2,2} +K_{2,2})
\left[ \left({1\over a_\#}
 +{1\over \sqrt{a_\#b_\#}}\right)
{(
\|{\bf f}\|_{p,\Omega}+ C_{n,p,r})|\Omega(k)|^{1/2-1/p}
\over (|\Omega(k)|+|\Gamma(k)|^{n/(n-1)})^\gamma} +\right.\nonumber\\
\left.+ \left({1\over b_\#}
 +{1\over \sqrt{a_\#b_\#}}\right){(
\|h\|_{r,\Gamma}+ C_{n,p,r})|\Gamma(k)|^{1/2-1/r}
\over (|\Omega(k)|+|\Gamma(k)|^{n/(n-1)})^\gamma} \right]\Sigma^\gamma.
\end{eqnarray*}
\item[Case $n=2$.]
Applying (\ref{energy}), we obtain
\begin{eqnarray*} 
I\leq (S_{1,1} +K_{1,1})
\left[ \left({|\Omega|^{1\over 2\alpha} \over a_\#}
 +{1\over \sqrt{a_\#b_\#}}\right)
{(
\|{\bf f}\|_{p,\Omega}+ C_{n,p,r})|\Omega(k)|^{{1\over 2}-{1\over p}}
\over (|\Omega(k)|+|\Gamma(k)|^{n/(n-1)})^\gamma} +\right.\nonumber\\
\left.+ \left({1\over b_\#}
 +{|\Omega|^{1/( 2\alpha)} \over \sqrt{a_\#b_\#}}\right){(
\|h\|_{r,\Gamma}+ C_{n,p,r})|\Gamma(k)|^{1/2-1/r}
\over (|\Omega(k)|+|\Gamma(k)|^{n/(n-1)})^\gamma} \right]\Sigma^{\gamma+1/
(2\alpha)}.
\end{eqnarray*}
\end{description}
 
In both cases, we infer from  (\ref{aalfa2}) that
\[ 
|\Sigma(h)|\leq \left( \mathcal{Z}_n
 \over h-k\right)^{\alpha}|\Sigma(k)|^\beta, 
\qquad\beta=\left\{\begin{array}{ll}
\alpha\gamma&\mbox{ if }n>2\\
\alpha\gamma+ 1/2&\mbox{ if }n=2
\end{array}\right.
\] 
where $\beta>1$ if and only if $p>n$ and $r>2(n-1)$.

By appealing to \cite[Lemma 4.1 (i)]{stamp63-64} we conclude
\[
|\Sigma(k_0+\mathcal{Z}_n|\Sigma(k_0)
|^{(\beta-1)/\alpha}2^{\beta/(\beta-1)})|=0.\]
This means that the essential supremmum does not exceed the well determined constant $k_0+
\mathcal{Z}_n|\Sigma(k_0)|^{(\beta-1)/\alpha}2^{\beta/(\beta-1)}$.
 This completes the proof of Proposition \ref{max}.
\end{proof}

\begin{remark}
 In particular, if $f=g=h=0$ on the corresponding domains
and ${\bf f} \in {\bf L}^p(\Omega)$ for $p>n$, then
\[
{\rm  ess } \sup_{\Omega\cup\partial\Omega}|u|\leq
1+Z_n\|{\bf f}\|_{p,\Omega}\left\{\begin{array}{ll}
(|\Omega|+|\partial\Omega|)^{{1\over n}-{1\over p}}
 \left({1\over a_\#}
 +{1\over \sqrt{a_\#b_\#}}\right)&\mbox{ if }n>2\\
(|\Omega|+|\partial\Omega|)^{{\alpha-1\over 2\alpha}-{1\over p}}
\left({|\Omega|^{1\over 2\alpha}\over a_\#}
 +{1\over \sqrt{a_\#b_\#}}\right)
&\mbox{ if }n=2
\end{array}\right.
\]
for every $\alpha>p/(p-2)$, with $Z_n=( S_{2,2}+K_{2,2})2^{n(p-2)\over 2(p-n)}$
if $n>2$, and $Z_2=( S_{1,1}+K_{1,1})
2^{(\alpha+1)/2-1/p\over( \alpha-1)/2-1/p}$.
\end{remark}
 
\subsection{Moser iteration technique}

\begin{proposition}\label{max0}
Let $p>n\geq 2$, $\ell\geq 2$, $u \in V_{2,\ell}$ 
be any weak solution to (\ref{omega})-(\ref{gama}),
in accordance with Definition \ref{def1}, under
 ${\bf f}\in  {\bf L}^{p}(\Omega)$,  $f\in L^{p/2}(\Omega)$,
 $g=0$ on $\Gamma_N$, and $h=0$ on $\Gamma$,
and (\ref{amin}) and (\ref{bmin}) be fulfilled.
 Then, $u$ satisfies 
\begin{equation}
{\rm  ess } \sup_{\Omega\cup\partial\Omega}|u|\leq
E_n\chi^{\left(\sum_{m\geq 0}m\chi^{-m}\right)}
(\sqrt{2}\mathcal{E})^{\chi/(\chi-1)}
\|u\|_{2p/(p-2),\Omega},
\label{cotaubound}
\end{equation}
with  $E_n=S_{2,2}^{\chi/(\chi-1)}$ and  $\chi=n(p-2)/[p(n-2)]$
if $n>2$, and $E_2= S_{p\chi/[p(\chi+1)-2],p\chi/[p(\chi+1)-2]}^{\chi/(\chi-1)}$
$\max\{|\Omega|^{(p-2)/[2p(\chi-1)]},|\Gamma|^{(p-2)/[2p(\chi-1)]}\}$
for any $\chi>1$, and
\[
\mathcal{E}=
\sqrt{(\|{\bf f}\|_{p,\Omega}^2/ a_\#
+2\|f\|_{p/2,\Omega})/ \min \{a_\#,b_\#\}}.
\]
\end{proposition}
\begin{proof} 
Let $\beta\geq 1$, and $k>1$. Defining the truncation operator
$T_k(y)=\min\{y,k\}$, set $w=T_k(|u|)\in H^{1}(\Omega)\cap
L^{\infty}(\Omega)$ that satisfies $w\in
L^{\infty}(\partial\Omega)$.
Choosing $v=\beta^2{\rm sign}(u)w^{2\beta-1}/(2\beta-1)\in
V_{2,\ell}$ as a 
test function in (\ref{pbu}), then $\nabla v=
\beta^2u^{2(\beta-1)}\nabla u$ in $\Omega[|u|<k]$,
and   $\nabla v= {\bf 0}$ in $\Omega\setminus\overline{\Omega[|u|<k]}$.
Thus, applying (\ref{amin}) and (\ref{bmin}) we deduce
\begin{eqnarray}\label{varab}
a_\#\int_{\Omega}|\nabla (w^\beta)|^2\mathrm{dx}+
 {\beta^2\over 2\beta-1}
b_\#\int_{\Gamma}|u|^{\ell-1}|w|^{2\beta-1}\mathrm{ds}\leq\\
\leq\int_{\Omega}    ( \mathsf{A}\nabla u)\cdot
\nabla v \mathrm{dx}+
\int_{\Gamma}b(u) v \mathrm{ds} \leq
{a_\#\over 2}\|\nabla (w^\beta)\|_{2,\Omega}^2+\nonumber\\
+{\beta^2\over 2a_\#}\|
{\bf f}w^{\beta-1}\|_{2,\Omega}^2+ {\beta^2\over 2\beta-1}
\|fw^{2\beta-1}  \|_{1,\Omega},\nonumber
\end{eqnarray}
using the H\"older inequality.

We may suppose that $w>1$. Otherwise, $w=|u|\leq 1<k$.
Using
 the H\"older inequality, we separately compute, for $p>2$, 
\begin{eqnarray*}
\int_{\Omega}|{\bf f}|^2w^{2(\beta-1)}  \mathrm{dx}\leq 
 \|{\bf f}\|_{p,\Omega}^2
\|w^{\beta}\|_{q,\Omega}^2,&&q=2p/(p-2)>2;
\\
\int_{\Omega}|f|w^{2\beta-1}  \mathrm{dx}\leq  
\|f\|_{{p/2},\Omega}
\|w^{\beta}\|_{q,\Omega}^2.&&
\end{eqnarray*}
Inserting these two inequalities in (\ref{varab}), and considering that
the left hand side absorbs the corresponding term of the right hand side,
 we obtain
\begin{equation}
\left(\|\nabla (w^\beta)\|_{2,\Omega}^2+\|w^\beta\|_{2,\Gamma}^2\right)^{1/2}
\leq  {\beta} \mathcal{E}\|w ^\beta\|_{q,\Omega}.\label{varaw}
\end{equation}

Let us split the proof of estimate (\ref{cotaubound}) into two space 
dimension  dependent cases.

{\sf Case} $n>2$.
Making use of (\ref{sob}), $w^\beta\in
W^{1,2}(\Omega)\hookrightarrow L^{q\chi}(\Omega)$, with $q \chi =
2n/(n-2)$ i.e. $\chi=n(p-2)/[p(n-2)]>1$ considering that $p>n$, and after
applying (\ref{varaw}), we deduce
\begin{eqnarray*}
\|w\|_{q\beta\chi ,\Omega}^\beta\leq S_{2,2}
\|w^\beta\|_{1,2,2}\leq S_{2,2}\sqrt{2}\left(
\|\nabla (w^\beta)\|_{2,\Omega}^2+
\|w^\beta\|_{2,\Gamma}^2\right)^{1/2}\\
\leq S_{2,2}\sqrt{2} \mathcal{E}\beta 
\|u\|_{q\beta,\Omega} ^\beta.
\end{eqnarray*}

Then,
we may pass to the limit the resulting inequality as $k\rightarrow\infty$
by Fatou lemma,
 obtaining
 \[
\|u\|_{q\beta\chi ,\Omega}\leq (\beta S_{2,2}\sqrt{2}
 \mathcal{E})^{1/\beta}
\|u\|_{q\beta,\Omega} .\]
 
Taking $\beta=\chi^m>1$, 
by induction, we have
\begin{equation}\label{tesei}
 \|u\|_{q\chi^{N},\Omega}\leq (S_{2,2}\sqrt{2} \mathcal{E})^{a_N}\chi^{b_N}
 \|u\|_{q,\Omega},\quad\forall N\in\mathbb{N},
\end{equation}
 where
 \[ a_N=\sum_{m=0}^{N-1}\chi^{-m}\quad
\mbox{and}\quad b_N=\sum_{m=0}^{N-1}m\chi^{-m}.
\]
Therefore, 
by the definition
 \[
 \|u\|_{\infty,\Omega}=
\lim_{N\rightarrow\infty} \|u\|_{q\chi^{N},\Omega},
 \]
 and observing that
 $\lim_{N\rightarrow\infty} a_N$ stands for the geometric series,
we find
\[
{\rm  ess } \sup_{\Omega}|u|\leq
E_n\chi^{\left(\sum_{m\geq 0}m\chi^{-m}\right)}(\sqrt{2}
\mathcal{E})^{\chi/(\chi-1)}\|u\|_{2p/(p-2),\Omega}.
\]

Next, making use of  (\ref{sobt}),
$w^{\beta}\in
W^{1,2}(\Omega)\hookrightarrow L^{2(n-1)/(n-2)}(\partial\Omega)$,
and (\ref{varaw}), we deduce
\begin{eqnarray*}
\|w\|_{\beta 2(n-1)/(n-2) ,\partial\Omega}^\beta\leq  K_{2,2}\sqrt{2}\left(
\|\nabla (w^\beta)\|_{2,\Omega}^2+
\|w^\beta\|_{2,\Gamma}^2\right)^{1/2} \leq \nonumber\\
\leq K_{2,2}\sqrt{2} \mathcal{E}\beta
\|u\|_{q\beta,\Omega} ^\beta .
\end{eqnarray*}
Taking $\beta=\chi^m>1$, and applying (\ref{tesei}), we get
\[
\|w\|_{\chi^m 2(n-1)/(n-2) ,\partial\Omega}
\leq K_{2,2}^{\chi^{-m}}S_{2,2}^{a_m}(\sqrt{2} \mathcal{E}
)^{a_{m+1}}\chi^{b_{m+1}}
 \|u\|_{q,\Omega}.\]
 Thus, 
we may pass to the limit the above inequality first as $k\rightarrow\infty$
by Fatou lemma, and next as $m\rightarrow\infty$, concluding
\[
{\rm  ess } \sup_{\partial\Omega}|u|\leq
E_n\chi^{\left(\sum_{m\geq 0}m\chi^{-m}\right)}(\sqrt{2}
\mathcal{E})^{\chi/(\chi-1)}\|u\|_{2p/(p-2),\Omega},
\]
which finishes (\ref{cotaubound}).

{\sf Case} $n=2$. 
Making use of $w^\beta\in
W^{1,2}(\Omega)\hookrightarrow
W^{1,2q\chi/(q\chi+2)}(\Omega)\hookrightarrow L^{q\chi}(\Omega)$, with $q\chi =2p\chi/(p-2)$ considering that $p>n=2$, and next applying
(\ref{varaw}), we deduce 
\begin{eqnarray*}
\|w\|_{q\beta\chi ,\Omega}^\beta\leq S_{2q\chi/(q\chi+2),2q\chi/(q\chi+2)}
\|w^\beta\|_{1,2q\chi/(q\chi+2),2q\chi/(q\chi+2)}
\leq\\
\leq S_{2q\chi/(q\chi+2),2q\chi/(q\chi+2)}\left(|\Omega|^{1/(q\chi)}
\|\nabla (w^\beta)\|_{2,\Omega}+|\Gamma|^{1/(q\chi)}
\|w^\beta\|_{2,\Gamma}\right) \leq \nonumber\\
\leq S_{2q\chi/(q\chi+2),2q\chi/(q\chi+2)}
\max\{|\Omega|^{1/(q\chi)},|\Gamma|^{1/(q\chi)}\}
\sqrt{2} \mathcal{E}\beta \|u\|_{q\beta,\Omega} ^\beta.
\end{eqnarray*}
For the boundary bound, we use $w^\beta\in
W^{1,2}(\Omega)\hookrightarrow
W^{1,2q\chi/(q\chi+1)}(\Omega)\hookrightarrow L^{q\chi}(\partial\Omega)$, 
deducing
\begin{eqnarray*} 
\|w\|_{q\beta\chi ,\partial\Omega}^\beta\leq 
K_{{2q\chi\over q\chi+1},{2q\chi\over q\chi+1}}
\max\{|\Omega|^{1/(2q\chi)},|\Gamma|^{1/(2q\chi)}\}
\sqrt{2} \mathcal{E}\beta \|u\|_{q\beta,\Omega} ^\beta.
\end{eqnarray*}
Thus, we may proceed as in the above case,
completing the proof of Proposition \ref{max0}.
\end{proof}

In  the following result stands for the particular case: $f=g=h=0$.
\begin{corollary}\label{ccinf}
Under the conditions of Proposition \ref{max0}, there exists
 a $L^\infty$-constant,  $C_\infty$, to the problem
(\ref{omega})-(\ref{gama}), that is, 
for $p>n$,
\[
{\rm  ess } \sup_{\Omega\cup\partial\Omega}|u|\leq
C_\infty \|{\bf f}\|_{p,\Omega}^{1+\chi/(\chi-1)},
\]
with
\begin{eqnarray*}
C_\infty=E_n\chi^{\left(\sum_{m\geq 0}m\chi^{-m}\right)}
\left({2\over a_\# \min \{a_\#,b_\#\}
}\right)^{\chi\over 2(\chi-1)}\times\\
\times  S_{{2pn\over 2p+n(p-2)},\ell}\left(
{|\Omega |^{{1\over n}-{2\over p}+{1\over 2}}
\over a_\#}
+\left({\ell\,'|\Omega|^{1-1/p}\over 2a_\# b_\#}\right)^{1/\ell}\right).
\end{eqnarray*}
\end{corollary}
\begin{proof} It suffices to insert  
the estimate (\ref{cotausup}) into (\ref{cotaubound}).
\end{proof}

\begin{proposition}\label{max2}
Let $n\geq 2$, $s>n-1$, $u \in H^{1}(\Omega)$ 
be any weak solution to (\ref{omega})-(\ref{gama}),
in accordance with Definition \ref{def1},
and (\ref{amin}) and (\ref{bmin}) be fulfilled.
If 
 ${\bf f}={\bf 0}$ and  $f=0$ in $\Omega$,
 $g\in L^{s}(\Gamma_N)$, and  $h\in L^{s}(\Gamma)$,
 then 
\begin{equation}
{\rm  ess } \sup_{\bar\Omega}|u|\leq
G_n\chi^{\left(\sum_{m\geq 0}m\chi^{-m}\right)}
(\sqrt{2}\mathcal{G})^{\chi/(\chi-1)}
 \|u\|_{2s/(s-1),\partial\Omega},\label{gh}
\end{equation}
with  $G_n=K_{2,2}^{\chi/(\chi-1)}$,  $\chi=(s-1)(n-1)/[s(n-2)]$
if $n>2$, and $G_2= K_{4s\chi/(2s\chi+s-1),4s\chi/(2s\chi+s-1)}^{\chi/(\chi-1)}$
$\max\{|\Omega|^{(s-1)/[4s(\chi-1)]},|\Gamma|^{(s-1)/[4s(\chi-1)]}\}$
for any $\chi>1$, and
\[
\mathcal{G}=\sqrt{\left(\|g\|_{s,\Gamma_N}+
 \|h\|_{s,\Gamma}
\right)/\min \{a_\#,b_\#\}}.
\]
\end{proposition}
\begin{proof} 
Let $\beta\geq 1$, and $k>1$. For $s>1$, and $q=2s/(s-1)$, 
proceeding as in the proof of Proposition \ref{max0}, we deduce
\begin{eqnarray*}
{a_\#}\int_{\Omega}|\nabla (w^\beta)|^2\mathrm{dx}+
 {\beta^2\over 2\beta-1}
b_\#\int_{\Gamma}|u|^{\ell-1}|w|^{2\beta-1}\mathrm{ds}\leq\nonumber\\
\leq  {\beta^2\over 2\beta-1}\left(
\|gw^{2\beta-1}\|_{1,\Gamma_N}
+ \|hw^{2\beta-1}\|_{1,\Gamma}\right)\leq\\
\leq\beta^2(\|g\|_{s, \Gamma_N} \|w^{\beta}\|_{q,\Gamma_N}^2+
\|h\|_{s, \Gamma} \|w^{\beta}\|_{q,\Gamma}^2).
\end{eqnarray*}
Thus, we obtain
\begin{equation}
\left(\|\nabla (w^\beta)\|^2_{2,\Omega}+
\|w^\beta\|_{2,\Gamma}^{2} \right)^{1/2} 
\leq  {\beta}\mathcal{G} \| w^\beta\|_{q,\partial\Omega},\label{varak2}
\end{equation}

Then,  we obtain
\begin{eqnarray*}
\|u\|_{q_1\beta\chi ,\Omega}\leq (\beta\sqrt{2}\mathcal{G}M_1)^{1/\beta}
\|u\|_{q\beta,\partial\Omega} ;\\
\|u\|_{q\beta \chi ,\partial\Omega}\leq (\beta\sqrt{2}\mathcal{G}
M_2)^{1/\beta}\|u\|_{q\beta,\partial\Omega},
\end{eqnarray*}
\begin{description}
\item[Case $n>2$] Setting $M_1= S_{2,2}$,
and $M_2= K_{2,2}$,
by using  (\ref{sob}), and
$w^{\beta}\in
W^{1,2}(\Omega)\hookrightarrow L^{q_1\chi}(\Omega)$,
 (\ref{sobt}), and
$w^{\beta}\in
W^{1,2}(\Omega)\hookrightarrow L^{q\chi}(\partial\Omega)$
with $q\chi=2(n-1)/(n-2)$ i.e. $\chi=(s-1)(n-1)/[s(n-2)]>1$
if $s>n-1$. 

\item[Case $n=2$] Setting $M_1=
S_{{2q_1\chi/( q_1\chi+2)},{2q_1\chi/(q_1\chi+2)}}
\max\{|\Omega|^{1/( q_1\chi)},|\Gamma|^{1/( q_1\chi)}\}$, and $M_2=
K_{{2q\chi/( q\chi+1)},{2q \chi/( q\chi+1)}}
\max\{|\Omega|^{1/(2q\chi)},|\Gamma|^{1/(2q\chi)}\}$, by using
 (\ref{sobt}) with
$w^{\beta}\in
W^{1,2q_1\chi/(q_1\chi+2)}(\Omega)\hookrightarrow L^{q_1\chi}(\Omega)$,
and (\ref{sobt}) with
$w^{\beta}\in
W^{1,2q\chi/(q\chi+1)}(\Omega)\hookrightarrow L^{q\chi}(\partial\Omega)$.
\end{description}
In both cases, 
 following the argument of the proof of Proposition \ref{max0}, we get
\begin{eqnarray*}
\|u\|_{q_1\chi ^{N+1},\Omega}\leq M_1^{\chi^{-N}}
(\sqrt{2}\mathcal{G})^{a_{N+1}}
M_2^{a_N}\chi^{b_{N+1}}
\|u\|_{q,\partial\Omega} ;&&\\
\|u\|_{q \chi ^N,\partial\Omega}\leq (\sqrt{2}\mathcal{G}
M_2)^{a_N}\chi^{b_N}\|u\|_{q,\partial\Omega},&&\quad\forall N\in\mathbb{N}.
\end{eqnarray*}
Therefore, we conclude (\ref{gh}), finishing
 the proof of Proposition \ref{max2}.
\end{proof}

\begin{corollary}[Linear Robin-Neumann problem]
Under the conditions of Propositions \ref{max0} and \ref{max2},
if (\ref{bstar}) is assumed then
there exists a weak solution,  $u\in H^1(\Omega)$, to (\ref{omega})-(\ref{gama})
in accordance with Definition \ref{def1}, such that
\begin{eqnarray}\label{cott}
{\rm  ess } \sup_{\bar\Omega}|u|\leq\Xi_1
\left({\|{\bf f}\|_{p,\Omega}^2/ a_\#
+2\|f\|_{p/2,\Omega}\over \min \{a_\#,b_*\}
}\right)^{\chi_1\over 2(\chi_1-1)}{\|{\bf f}\|_{2,\Omega}
+{L}_n\|f\|_{t,\Omega}\over \min \{a_\#,b_*\}}+\\ 
+\Xi_2 \left({\|g\|_{s,\Gamma_N}+
 \|h\|_{s,\Gamma}\over \min \{a_\#,b_*\}}
\right)^{\chi_2\over 2(\chi_2-1)}
{ {M}_n(\|g\|_{s,\Gamma_N}+
 \|h\|_{s,\Gamma})\over \min \{a_\#,b_*\}
},\nonumber
\end{eqnarray}
where
\begin{eqnarray*}
\Xi_1&=&
E_n\chi_1^{\left(\sum_{m\geq 0}m\chi_1^{-m}\right)}
\sqrt{2}^{\chi_1/(\chi_1-1)}
  S_{{2pn\over 2p+n(p-2)},\ell}\left(|\Omega |^{p-n\over np}+1\right);\\
\Xi_2&=&
G_n\chi_2^{\left(\sum_{m\geq 0}m\chi_2^{-m}\right)}
\sqrt{2}^{\chi_2/(\chi_2-1)}
  K_{{2sn\over 2s+(n-1)(s-1)},\ell}\left(|\Omega |^{s-n+1\over 2ns}+1\right),
\end{eqnarray*}
with 
$E_n$ and $\chi_1$ being the constants in accordance with Proposition \ref{max0},
$G_n$ and $\chi_2$ being the constants in accordance with Proposition \ref{max2},
$L_n=2S_{2,\ell}$ if $n>2$, $L_2=(|\Omega|^{1/t'}+1)S_{2t/(3t-2),\ell}$
 if $t<2$, $L_2=(|\Omega|^{1/2}+1)|\Omega|^{1/2-1/t}S_{1,\ell}$ if $t\geq 2$,
 $M_n=2K_{2,\ell}$ if 
$n>2$, and $M_2=(|\Omega|^{1/(2s')}+1)K_{2s/(2s-1),\ell}$.
\end{corollary}
\begin{proof}
 From Propositions \ref{exist} and \ref{max0},
 there exists $u_1\in  H^1(\Omega)$ solving
\[
\int_{\Omega}    ( \mathsf{A}\nabla u_1)\cdot
\nabla v \mathrm{dx}+\int_{\Gamma}u_1 v \mathrm{ds}
=\int_{\Omega}{\bf f}\cdot\nabla v \mathrm{dx}
+\int_{\Omega}fv \mathrm{dx},
 \quad\forall v\in  H^1(\Omega),
 \]
such that it verifies
\begin{equation}
{\rm  ess } \sup_{\bar\Omega}|u_1|\leq
E_n\chi_1^{\left(\sum_{m\geq 0}m\chi_1^{-m}\right)}
(\sqrt{2}\mathcal{E})^{\chi_1/(\chi_1-1)}
\|u_1\|_{2p/(p-2),\Omega}.\label{pbuff}
\end{equation}
 From Propositions \ref{exist} and \ref{max2}, and Remark \ref{sol},
 there exists $u_2\in H^1(\Omega)$ solving
\[
\int_{\Omega}    ( \mathsf{A}\nabla u_2)\cdot
\nabla v \mathrm{dx}+\int_{\Gamma}u_2 v \mathrm{ds}
=\int_{\Gamma_N} gv \mathrm{ds}+\int_{\Gamma} hv \mathrm{ds},
 \quad\forall v\in  H^1(\Omega),
 \]
such that it verifies
\begin{equation}
{\rm  ess } \sup_{\bar\Omega}|u_2|\leq
G_n\chi_2^{\left(\sum_{m\geq 0}m\chi_2^{-m}\right)}
(\sqrt{2}\mathcal{G})^{\chi_2/(\chi_2-1)}
 \|u_2\|_{2s/(s-1),\partial\Omega}.\label{pbugh}
\end{equation}
 
Then,   $u=u_1+u_2\in H^1(\Omega)$ is the required solution.
Moreover, from (\ref{pbuff})-(\ref{pbugh}) gathered with Corollary \ref{coruq}
we find (\ref{cott}), with ${L}_nA=\mathcal{F}_n(A,0)+\mathcal{H}_n(A,0)$, 
${M}_n(B)=\mathcal{F}_n(0,B)+\mathcal{H}_n(0,B)$, and
$\mathcal{F}_n$ and $\mathcal{H}_n$ being the functions 
 in accordance with Proposition \ref{exist}.
\end{proof}

\section{$V_{q,\ell-1}$-solvability $(q< n/(n-1)$,
$\ell\geq 2)$}
\label{l1data}

The $W^{1,q}$-solvability depends on the data regularity.
In the presence of the boundary condition (\ref{robin}),
the duality theory is more straightforward  than the $L^1$-theory
when $L^1$ data are taken into account.
In  order to determine the explicit constant,
 the following result of
the existence of a  solution is based on  the duality theory.

First let us recall that, for $q>1$, the $L^q$-norm may be defined as
\begin{equation}\label{dnorm}
\|{\bf u}\|_{q,\Omega}=\sup\left\{
|\int_{\Omega}{\bf u}\cdot {\bf g}\mathrm{dx}|
: \ {\bf g}\in {\bf L}^{q'}(\Omega), {\|{\bf g}\|
_{q',\Omega}}=1
\right\},
\end{equation}
for all ${\bf u}\in{\bf L}^q(\Omega)$.

\begin{proposition}\label{W1q}
Let  ${\bf f}={\bf 0}$ a.e. in $\Omega$,
 $f\in L^1(\Omega)$,  $g\in L^1(\Gamma_N)$,
$h\in L^1(\Gamma)$, (A)-(B) be fulfilled,  and $\mathsf{A}$ be symmetric.
For any $\ell \geq 2$, there exists $u\in
  V_{q,\ell -1}$ solving (\ref{pbu}) for every $1< q < n/(n-1)$.
Moreover, we have the following estimate
\begin{eqnarray}\label{cota1q}
\| \nabla u \|_{q,\Omega}&\leq &C_\infty\Big(|\Gamma|(1+b^\#)
+  \| f\|_{1,\Omega}+\|g\|_{1,\Gamma_N}+\|h\|_{1,\Gamma}+\\
&+&(1+b^\#)\left(  \| f\|_{1,\Omega}+\|g\|_{1,\Gamma_N}+\|h\|_{1,\Gamma}
\right)/b_\#\Big); \nonumber\\ 
\| u\|_{\ell-1,\Gamma}^{\ell-1}&\leq &|\Gamma| + \left(  \| f\|_{1,\Omega}+\|g\|_{1,\Gamma_N}+\|h\|_{1,\Gamma}
\right)/
b_\#,\label{cota1ql}
\end{eqnarray}
with
the constant $C_\infty$ being explicitly 
given in Corollary \ref{ccinf}.
\end{proposition}
\begin{proof}
For each $m\in \mathbb N$, take $
f_m=F_m(f)\in L^\infty(\Omega),$
$g_m=F_m (g)\in L^\infty(\Gamma_N),$
$h_m=F_m (h)\in L^\infty(\Gamma)$, with
\[
F_m(\tau)={m \tau\over m+|\tau |}.
\]
Applying Proposition \ref{exist},
there exists a unique solution $u_m\in V_{2,\ell}
$ to the following variational problem
\begin{eqnarray}\label{pbum}
\int_\Omega  (\mathsf{A}\nabla u_m)\cdot\nabla v\mathrm{dx}
+\int_\Gamma b(u_m) v\mathrm{ds}=
\int_\Omega f_mv\mathrm{dx}+\\
+\int_{\Gamma_N}g_mv\mathrm{ds}+\int_\Gamma
h_mv\mathrm{ds},\quad\forall v\in V_{2,\ell}.\nonumber
\end{eqnarray}
In particular, (\ref{pbum}) holds for all $v\in W^{1,q'}(\Omega)$ for $q'>n$.
 Defining the truncation operator
$T_1(y)={\rm sign}(y)\min\{|y|,1\}$, let us
choose $v=T_1(u_m)\in V_{2,\ell}$ as a 
test function in (\ref{pbum}), obtaining
\begin{equation}\label{gum}
b_\#\left(\int_{ \Gamma [ |u_m|>1]}|u_m|^{\ell-1}\mathrm{ds}+
\int_{ \Gamma [ |u_m|\leq 1]}|u_m|^{\ell}\mathrm{ds}\right)\leq
\| f\|_{1,\Omega}+\| g\|_{1,\Gamma_N}+\|h\|_{1,\Gamma}.
\end{equation}
Hence, we conclude that (\ref{cota1ql}) is true for $u_m$.

In order to pass to the limit (\ref{pbum}) on $m$ $(m\rightarrow\infty)$
let us establish the estimate (\ref{cota1q}) for
 $  u_m$.
Let $w\in V_{2,2}$ be the unique  weak solution to the mixed
Robin-Neumann problem (\ref{omega})-(\ref{gama}),
under $f=g=h=0$,
in accordance with Proposition \ref{exist}.
 Since $\mathsf{A}$ is symmetric, we infer that
\[ 
\int_{\Omega}    ( \mathsf{A}\nabla u_m)\cdot
\nabla w \mathrm{dx}=
\int_{\Omega}    ( \mathsf{A}\nabla w)\cdot
\nabla u_m \mathrm{dx}
=\int_{\Omega}{\bf f}\cdot\nabla u_m \mathrm{dx}-
\int_{\Gamma}w u_m \mathrm{ds},
\] 
which gathered with (\ref{pbum}) under $v=w$ reads
\begin{eqnarray}\label{fum}
\int_{\Omega}{\bf f}\cdot\nabla u_m \mathrm{dx}=
\int_{\Gamma}w u_m \mathrm{ds}-\int_{\Gamma}b(u_m) w \mathrm{ds}
+\int_{\Omega}f_mw \mathrm{dx}+\\
+\int_{\Gamma_N}g_mw \mathrm{ds}+\int_{\Gamma}h_mw \mathrm{ds}.\nonumber
\end{eqnarray}

For ${\bf f}\in {\bf L}^{q'}(\Omega)$ with $q'>n$ such that
$\|{\bf f}\|_{q',\Omega}=1$,
Corollary \ref{ccinf} guarantees that the existence of a $L^\infty$-constant
 $C_\infty$ such that
$\|w\|_{\infty,\Omega}+
\|w\|_{\infty,\partial\Omega}\leq C_\infty $.
 
By (\ref{dnorm}) with ${\bf u}=\nabla u_m $, and (\ref{fum}), we obtain 
\begin{eqnarray*}
\|\nabla u_m\|_{q,\Omega}\leq C_\infty\Big( |\Gamma|
(1+b^\#)+(1+b^\#)
\|u_m\|_{\ell-1,\Gamma[|u_m|>1]}^{\ell-1}+\\
+ \| f\|_{1,\Omega}+\|g\|_{1,\Gamma_N}+\|h\|_{1,\Gamma}
\Big).
\end{eqnarray*}
Applying  (\ref{gum}), then (\ref{cota1q}) holds for $u_m$.

Therefore, the passage to the limit as $m$ tends to infinity
is allowed, concluding the proof of Proposition \ref{W1q}.
\end{proof}

\section{Green kernel}
\label{secg}

In this Section, we establish the existence  of the Green kernel
altogether some of its properties. Here, we follow the approach introduced in 
\cite{gw} in constructing Green's function for
 the Dirichlet problem (see also \cite{lsw}).
\begin{definition}\label{def2}
For each $x\in\Omega$, we say that $G= G(x,\cdot)$ is a Green kernel associated
 to (\ref{omega})-(\ref{gama}), 
if it solves, in the distributional sense,
\begin{eqnarray}\label{pbud}
 \nabla\cdot(\mathsf{A}\nabla  G(x,\cdot))=\delta_x\quad\mbox{ in }\Omega;\\
 \mathsf{A}\nabla  G(x,\cdot)\cdot {\bf n}+b( G(x,\cdot))\chi_\Gamma=0
\quad\mbox{ on }\partial\Omega,
 \end{eqnarray}
 where $\delta_x$ is the Dirac delta function at the point
$x$. That is, there is $q>1$
such that $G$ verifies the variational formulation
\begin{equation}\label{varf}
\int_{\Omega}     \mathsf{A}\nabla G(x,\cdot)\cdot
\nabla v \mathrm{dy}+\int_\Gamma b(G(x,\cdot))v\mathrm{ds}
=v(x),\quad\forall v\in V_{q,\ell}.
\end{equation}
\end{definition}

\begin{proposition}\label{green}
Let  $n\geq 2$,  $1\leq  q<n/(n-1)$,
(A)-(B) be fulfilled,  and $\mathsf{A}$ be symmetric.
Then, for each $x\in\Omega$ and any $r>0$
such that   $r< {\rm dist}(x,\partial\Omega)$,
 there exists a unique  Green function $G=
G(x,\cdot)\in V_{q,\ell-1}\cap H^1(\Omega \setminus B_r(x))$ 
 according  to Definition \ref{def2},
and  enjoying the following estimates
\begin{eqnarray}
\label{g1}
\|\nabla G\|_{q,\Omega} &\leq &
C_\infty \left(1+(1+b^\#)(|\Gamma|+1/b_\#)
\right);\\
 \| G\|_{\ell-1,\Gamma} &\leq &( |\Gamma|+1/b_\#)^{1/(\ell-1)}
,\label{gg}
\end{eqnarray}
with
the constant $C_\infty$ being explicitly 
given in Corollary \ref{ccinf}.
Moreover, $G(x,y)\geq 0$ a.e.  $x,y\in \Omega$.
\end{proposition}
\begin{proof}
Let $x\in\Omega$ and $\rho>0$ be such that $B_\rho(x)\subset\subset\Omega$.
Thanks to Proposition \ref{exist} with  ${\bf f}={\bf 0}$ a.e. in $\Omega$,
$g,h=0$ a.e. on, respectively, $\Gamma_N$ and $\Gamma$,
 and $f=\chi_{B_\rho(x)}/|B_\rho(x)| $ belonging to
$ L^{2n/(n+2)}(\Omega)$ if 
$n>2$, and to $L^{2}(\Omega)$ if $n=2$,
there exists
$G^\rho=G^\rho(x,\cdot)\in V_{2,\ell}$ being the unique solution to
\begin{equation}\label{pbgreen}
\int_{\Omega}     \mathsf{A}\nabla G^\rho\cdot
\nabla v \mathrm{dy}+\int_\Gamma b(G^\rho)v\mathrm{ds}
={1\over |B_\rho(x)|}\int_{B_\rho(x)}v \mathrm{dy},
\end{equation}
for all $v\in V_{2,\ell} $.
In particular, if $\ell=2$  (\ref{cotau}) reads
\begin{equation}\label{grho2}
\|\nabla G^\rho\|_{2,\Omega}+
\| G^\rho\|_{2,\Gamma}\leq{{2}\over\min\{ a_\#,b_\#\}}\times\left\{
\begin{array}{ll}
 S_{2,2} \omega_n^{{1\over n}-{1\over 2}}\rho^{1-n/2}&\mbox{ if } n>2\\
\sqrt{  |\Omega |+1\over \pi} S_{1,2}
\rho^{-1}&\mbox{ if } n=2
 \end{array}\right. .
\end{equation}
Therefore, for any
 $r>0$ such that $B_r(x)\subset\subset\Omega$,
 there exists $G\in H^1(\Omega\setminus B_r(x))$
such that
\[
G^\rho\rightharpoonup G\quad \mbox{in } H^1(\Omega\setminus B_r(x))\quad
\mbox{ as $\rho\rightarrow  0^+$.}
\]

In order to $G$ verify
\begin{equation}\label{defgreen}
G(x,y)=\lim_{\rho\rightarrow 0}
G^\rho(x,y),\quad
\forall x\in\Omega, \ \mbox{a.e. }y\in\Omega, \ x\not=y,
\end{equation}
we observe that
the $V_{q,\ell-1}$-estimates  (\ref{g1})-(\ref{gg})
 are true for $G^\rho$ due to
(\ref{cota1q})-(\ref{cota1ql}), by
applying Proposition \ref{W1q} with
$g,h=0$, and $f=\chi_{B_\rho(x)}/|B_\rho(x)|\in L^{1}(\Omega)$.
 Then, we can extract a subsequence of $G^\rho,$ still denoted by
$G^\rho,$ 
weakly converging to $G$ in $V_{q,\ell-1}$
as $\rho$ tends to $0$, with $G\in V_{q,\ell-1}$
 solving (\ref{varf}) for all $v\in W^{1,q'}(\Omega)$.
A well-known property of passage to the weak limit implies
 (\ref{g1})-(\ref{gg}).
 
In order to prove the nonnegativeness assertion, first
calculate
\[a_\#
\int_{\Omega}     |\nabla (G^\rho-|G^\rho|)|^2\mathrm{dy}
\leq {2\over |B_\rho(x)|}
\left(\int_{B_\rho(x)}G^\rho \mathrm{dy}-\int_{B_\rho(x)}|G^\rho| 
\mathrm{dy}\right)\leq 0.
\]
Then, $G^\rho=|G^\rho|$, and by passing to the limit as $\rho$ tends to $0$,
the nonnegativeness claim holds,
which completes the proof of Proposition \ref{green}.
\end{proof}

\section{Robin-Neumann problem ($\ell=2$)}
\label{lqc}

 In the two-dimensional space, Proposition \ref{exist}
 leads $H^1$ solution for  the $L^p$-data, with an arbitrary $p>1$.
 Our concern is then the existence of weak solutions and the derivation of
their estimates in the $n$-dimensional space: $n>2$.
\begin{proposition}\label{W1qp}
Let  ${\bf f}={\bf 0}$ a.e. in $\Omega$,
 $f\in L^t(\Omega)$ with  $t\leq 2n/(n+2)$,  $g\in L^s(\Gamma_N)$ and
$h\in L^s(\Gamma)$ with
 $s\leq 2(n-1)/n$,   and $\mathsf{A}$ be a symmetric
matrix satisfying the assumption (A).
Under the assumption (\ref{bstar}) with $b_*=1$,
 there exists $u\in W^{1,q}(\Omega)$
 solving (\ref{pbu}) for every $1< q < 2(n-1)p/[2(n-1)-p]$
with $p=\min\{t,s\}$.
Moreover, we have the following estimate
\begin{equation}
\| \nabla u \|_{q,\Omega}
\leq\mathcal{M}_q\left({1\over a_\#}+{1\over\sqrt{a_\#}}
\right)\left(
\mathcal{K}_{t',q'}
 \| f\|_{t,\Omega}+\mathcal{K}_{s',q'}(\|g\|_{s,\Gamma_N}+\|h\|_{s,\Gamma})
\right),
\label{li}
\end{equation}
with
\begin{eqnarray*}
\mathcal{M}_q=  |\Omega|^{(n-2)/[2(n-1) n]} S_{2,2}+K_{2,2}
+\\
+2|\Omega|^{{1\over q}-{1\over 2}}
\left( S_{{n\over n+1},{n\over n+1}}(|\Omega|^{{1\over 2}+{1\over n}}
+|\partial\Omega|^{{1\over 2}+{1\over n}})+K_{1,1}
(|\Omega|^{1\over 2}+
|\partial\Omega|^{{1\over 2}})\right)
.\end{eqnarray*}
\end{proposition}
\begin{proof}
For each $m\in \mathbb N$, take the approximations
$f_m$, $g_m$, and $h_m$ as in the proof of Proposition \ref{W1q},
and the corresponding unique solution $u_m\in V_{2,2}
$ to the  variational problem (\ref{pbum}).
Moreover, (\ref{cota1ql}) is true for $u_m$ ($\ell=2$).

Let $w\in V_{2,2}$ be the unique  weak solution to the mixed 
Robin-Neumann problem (\ref{omega})-(\ref{gama}),
under $f=g=h=0$, such that  (\ref{fum}) reads
\begin{equation}
\int_{\Omega}{\bf f}\cdot\nabla u_m \mathrm{dx}=
\int_{\Omega}f_mw \mathrm{dx}
+\int_{\Gamma_N}g_mw \mathrm{ds}+\int_{\Gamma}h_mw \mathrm{ds}.\label{umfl}
\end{equation}
Moreover, for $q'\geq 2$ (\ref{cotau}) reads
\begin{equation}\label{cotau2}
\|w\|_{1,2,2}\leq \left({1\over a_\#}+{1\over\sqrt{a_\#}}
\right)|\Omega|^{1/q-1/2}
\|{\bf f}\|_{q',\Omega}.
\end{equation}
Observe that
\begin{eqnarray}\label{ww}
\|w\|_{1,\Omega}+\|w\|_{1,\partial\Omega}\leq 
\left(S_{{n\over n+1},{n\over n+1}}\left(|\Omega|^{1/2+1/n}
+|\Gamma|^{1/2+1/n}\right)+\right. \\
\left.+K_{1,1}\left(|\Omega|^{1/2}
+|\Gamma|^{1/2}\right)\right)\|w\|_{1,2,2}.\nonumber
\end{eqnarray}

For any $t\leq 2n/(n+2)$, $s\leq 2(n-1)/n$,
 and $q<2(n-1)p/[2(n-1)-p]$ with $p=\min\{t,s\}$, which means
 $2n/(n-2)\leq t'<2(n-1)q'/[2(n-1)-q']$, and 
 $2(n-1)/(n-2)\leq s'<2(n-1)q'/[2(n-1)-q']$ if $2<q'<2(n-1)$,
 Proposition \ref{qestimates} (with $\delta=1/2-1/q'$ since $h\equiv 0$)
 yields
\begin{eqnarray*}
\|w\|_{t',\Omega}
\leq \mathcal{K}_{t',\delta}\left(
 (|\Omega|^{n-2\over 2(n-1) n} S_{2,2}+K_{2,2})
\left({1\over a_\#}+{1\over\sqrt{a_\#}}\right)
 \|{\bf f}\|_{q',\Omega}+\right.\\
\left.+2(
\|w\|_{1,\Omega}+\|w\|_{1,\partial\Omega})
\right)
\leq  \mathcal{K}_{t',\delta}\mathcal{M}_q
\left({1\over a_\#}+{1\over\sqrt{a_\#}}\right)
 \|{\bf f}\|_{q',\Omega},
\end{eqnarray*}
considering that $(2(n-1)-q')/[2(n-1)q']<1$ is taken into account,
and applying (\ref{ww}) 
accomplished with (\ref{cotau2}).
Analogously
\[
 \| w\|_{s',\partial\Omega}
\leq  \mathcal{K}_{s',\delta}\mathcal{M}_q
\left({1\over a_\#}+{1\over\sqrt{a_\#}}\right)
 \|{\bf f}\|_{q',\Omega}.
 \]

By (\ref{dnorm}) with ${\bf u}=\nabla u_m $,  we infer from
(\ref{umfl}) that (\ref{li})
holds for $u_m$. Therefore, by passage to the limit as $m$ tends to
infinity, we conclude the claimed result.
\end{proof}

\end{document}